\documentclass[a4paper,10pt]{article}

\usepackage{geometry}
\usepackage{amsmath}
\usepackage{amssymb}
\usepackage{amsthm}
\usepackage{esint}
\allowdisplaybreaks

\usepackage{xifthen}

\usepackage{aliascnt} 
\usepackage[bookmarksopen,bookmarksdepth=2,backref=none,hidelinks]{hyperref}

\usepackage{comment}

\theoremstyle{definition}
\newtheorem{definition}{Definition}[section]
\newaliascnt{lemma}{definition}
\newaliascnt{theorem}{definition}
\newaliascnt{corollary}{definition}
\newaliascnt{proposition}{definition}
\newaliascnt{remark}{definition}
\newaliascnt{example}{definition}
\newaliascnt{conjecture}{definition}

\newtheoremstyle{remark2}{}{}{}{}{\bfseries}{.}{ }{}
\theoremstyle{remark2}
\newtheorem{examplex}[example]{Example}
\newtheorem{remarkx}[remark]{Remark}
\newenvironment{remark}
  {\pushQED{\qed}\remarkx}
  {\popQED\endremarkx}
\newenvironment{example}
  {\pushQED{\qed}\examplex}
  {\popQED\endexamplex}

\theoremstyle{plain}
\newtheorem{lemma}[lemma]{Lemma}
\newtheorem{theorem}[theorem]{Theorem}

\newtheorem{proposition}[proposition]{Proposition}

\aliascntresetthe{lemma}
\aliascntresetthe{theorem}
\aliascntresetthe{corollary}
\aliascntresetthe{proposition}
\aliascntresetthe{remark}
\aliascntresetthe{conjecture}

\newcommand{\R}{\ensuremath{\mathbb{R}}}

\newcommand{\N}{\ensuremath{\mathbb{N}}}

\newcommand{\Z}{\ensuremath{\mathbb{Z}}}
\newcommand{\abs}[1]{\left|{#1}\right|} 
\newcommand{\norm}[2][]{\left\|#2\right\|_{#1}} 

\renewcommand{\S}{\ensuremath{\mathbb{S}}}
\newcommand{\T}{\ensuremath{\mathbb{T}}}
\DeclareMathOperator{\supp}{supp}

\title{A space-time relaxation for $L^1$ optimal control problems}
\author{Malte Kampschulte\footnote{Department of mathematical analysis, Faculty of mathematics and physics, Charles University Prague. Email: \nolinkurl{kampschulte@karlin.mff.cuni.cz}}}

\begin{document}
\maketitle
\begin{abstract}
 We introduce a vertical type relaxation for optimal control problems which only have $L^1$-coercivity for their controls. Usually such problems feature both concentration and oscillation effects at the same time. We propose relaxing to an associated problem in space-time, where the controls can be considered bounded in $L^\infty$, greatly simplifying any analysis. In this relaxation, concentrations are transformed into vertical parts and oscillations can be dealt with using Young-measures. This technique can be extended to similar problems on infinite-dimensional spaces.
\end{abstract}

\section{Introduction}

Consider an optimal control problem of the general type
\begin{align*}
 \text{Minimize } & \mathcal{F}(y,u) := \int_0^T f(t,y(t),u(t)) dt + g(y(T))\\
 \text{where } & \dot{y}(t) = A(t,y(t),u(t)) \quad \forall t\in [0,T]\\
 &y(0) = y_0
\end{align*}
where $f$ is growing asymptotically linear in $u$, that is $\lim_{\lambda\to\infty} \frac{f(t,y,\lambda u)}{\lambda u} = c(t,y,u)$ with $0 < c_{\text{min}} < c(t,y,u) < c_{\text{max}} < \infty$ for $\abs{u}=1$. In this case an obvious choice of space to work with is $u\in L^1([0,T];\R^k)$ and indeed a minimizing sequence $u_k$ will be bounded and coercive in $L^1$. However, as $L^1$ is not reflexive, we will not get a weakly converging subsequence in $L^1$. Instead there will be concentrations and as a result, the class $L^1([0,T];\R^k)$ does not have to have a minimizer. In fact there might be different concentrations at the same time $t$ which do not cancel due to the nonlinearities of the problem. 

All this shows the need for a proper relaxation of $u$ which can both capture those concentrations as well as any oscillation effects that need to be taken care of if the problem lacks convexity. Indeed, such a relaxation is what Kružík and Roubíček developed in \cite{KruzikRoubicek99}, where they used DiPerna-Majda measures \cite{DiPernaMajda87} to encapsulate both effects at the same time. In this paper, we propose an alternative approach, which is both simpler in terms of analysis as well as more fine grained in its recovery of concentration aspects of the problem.

\subsection{\texorpdfstring{$L^1$}{L1} optimal control problems: Developing a relaxation}

To introduce this approach, let us start by carrying it out on an example.

\begin{example}[Inspired by {\cite[Ex.\@ 1.2]{KruzikRoubicek99}}] \label{ex:updown}
\begin{align*}\text{Minimize } &\mathcal{F}(y,u):= \int_0^T ((t-1)^2+1)\abs{u} + \abs{y_2}^2 dt +(y_1(T)-1)^2\\
\text{where }&
  \dot{y}_1 = \abs{u},\, y_1(0) = 0\\
  &\dot{y}_2 = u, \, y_2(0) = 0
\end{align*}
and $u \in L^1([0,T]), y \in W^{1,1}([0,T];\R^2)$ for a fixed $T>1$.
\end{example}

This problem puts several obvious constraints on $u$. Due to the final term in the energy, $y_1$ wants to have a final value $y_1(T)$ close to $1$, so integration of the ODE yields $\int_0^T \abs{u} dt \approx 1$. In the same way, by the second term, $\abs{y_2}$ will want to stay small, so the average $\int_0^{t_0} u(t) dt$ has to stay close to $0$. This alone could be satisfied by a simple sequence of oscillating functions, resulting in a Young measure as limit. However, then there is also the first term in the energy, which strongly prefers $\abs{u}$ to concentrate around $t=1$.

It is possible to combine all those constraints into a single sequence and show that it is minimizing. Fix $c>0$ and define
\[u_k: t\mapsto \begin{cases} \frac{c}{\delta_k} &\text{ for } 1-\delta_k \leq t < 1\\ -\frac{c}{\delta_k} &\text{ for } 1 \leq t < 1+\delta_k \\ 0 &\text{ otherwise.}\end{cases}\]
Then solving the ODE shows $(y_k)_1(T) = 2c$ as well as $(y_k)_2 = 0$ outside of $[1-\delta_k,1+\delta_k]$ and $\abs{(y_k)_2} \leq c$ in $[1-\delta_k,1+\delta_k]$. Thus
\[\mathcal{F}(u_k,y_k) \leq \int_{1-\delta}^{1+\delta} (\delta^2+1)\frac{c}{\delta} + c^2dt +(2c-1)^2 \to 2c+(2c-1)^2.\]
Equally there is a lower estimate
\[\mathcal{F}(u,y) \geq \int_0^T \abs{u} dt + (y(T)-1)^2 \geq y_1(T) + (y_1(T)-1)^2\]
so by taking the optimal choice $c := \frac{y(T)}{2} = \frac{1}{4}$ we have indeed a minimizing sequence.

Now what does this mean in the limit? The control will concentrate at $t=1$, but both with positive and negative sign at the same time. Tshis is something which is not possible in any ordinary linear sense such as with distributions, but is a behaviour that needs to be captured using some relaxation. Kružík and Roubíček use DiPerna-Majda measures for this, but in order to develop a different way, let us shift the focus from $u$ towards $y$.

What $y_1$ does at $t=1$, is to move quickly from $0$ to $1/2$ while $y_2$ moves up to $1/4$ and then immediately down to $0$ again. If we continue to model $y$ as a curve in space, those jumps are also problematic. However if we consider them as a curve in space-time instead, they turn into a mostly harmless object. We get a curve $(\tilde{t},\tilde{y})$ in $[0,T]\times\R^2$ that starts at $(0,0,0)$ and then moves in straight line segments through the points $(1,0,0)$, $(1,1/4,1/4)$, $(1,1/2,0)$ and $(T,1/2,0)$.

The goal now is to reformulate the original problem into a related one that uses this curve instead. For this, we will parameterize our new curve as 
\[(\tilde{t},\tilde{y}): [0,S] \to [0,T] \times \R^2; s\mapsto (\tilde{t}(s),\tilde{y}(s)).\] 
We immediately get the boundary conditions $\tilde{y}(0) = 0$, $\tilde{t}(0) = 0$ and $\tilde{t}(S) = T$. 

Concerning the ODE, we want to relate to the original $y$ via $\tilde{y}(s) = y(\tilde{t}(s))$, so if we multiply the original equation by $\frac{dt}{ds}$, we can apply the chain rule to get
\begin{align*}
 \frac{d\tilde{y}_1}{ds}(s) &= \frac{dy_1}{dt}(\tilde{t}(s))\frac{d\tilde{t}}{ds}(s) = \abs{u(\tilde{t}(s))} \frac{d\tilde{t}}{ds}(s) \\
 \frac{d\tilde{y}_2}{ds}(s) &= \frac{dy_2}{dt}(\tilde{t}(s))\frac{d\tilde{t}}{ds}(s) = u(\tilde{t}(s)) \frac{d\tilde{t}}{ds}(s).
\end{align*}
At this point it is reasonable to postulate some control on the flow of time. On one hand $\tilde{t}$ should never be decreasing, as time cannot move backwards. On the other hand, as our fundamental object is a curve and not its parametrization, it does not matter how fast we move through time. As such we can introduce another control variable $\tilde{v}:[0,S] \to [0,\infty)$ along with the differential equation $\frac{d\tilde{t}}{ds}(s) = \tilde{v}(s)$. Now whenever $u$ gets big, we can keep $\frac{d\tilde{y}}{ds}$ small by decreasing $\tilde{v}$. And if we want to have a vertical part, we can set $\tilde{v}(s)=0$ on an interval. To better utilize this time rescaling, it is convenient to define a new control variable $\tilde{u}:[0,S] \to \R^2$ which will correspond to $u(\tilde{t}(s))v(s)$. Then in total we will get the following relaxed ODEs:
\begin{align*}
 \frac{d\tilde{y}_1}{ds}(s) &= \abs{\tilde{u}(s)}\\
 \frac{d\tilde{y}_2}{ds}(s) &= \tilde{u}(s)\\
 \frac{d\tilde{t}}{ds}(s) &= \tilde{v}(s)
\end{align*}

Furthermore as any large values of $u$ can be reached by setting $\tilde{v}$ small, we can bound our controls in $L^\infty$ by $\abs{\tilde{u}(s)} \leq 1$ and $0\leq \tilde{v}(s)\leq 1$. Finally we need to transform the energy by a simple change of variables:

\begin{align*}
 &\phantom{{}={}} \int_0^T ((t-1)^2+1)\abs{u}(s) + \abs{y_2}^2 dt +(y_1(T)-1)^2\\
 &= \int_0^S \left[((\tilde{t}(s)-1)^2+1)\frac{\abs{\tilde{u}(s)}}{\tilde{v}(s)} + \abs{\tilde{y}_2(s)}^2 \right]\frac{d\tilde{t}}{ds} ds +(\tilde{y}_1(S)-1)^2\\
 &= \int_0^S ((\tilde{t}(s)-1)^2+1)\abs{\tilde{u}(s)} + \abs{\tilde{y}_2(s)}^2 \tilde{v}(s) ds +(\tilde{y}_1(S)-1)^2 =: \tilde{\mathcal{F}}(\tilde{t},\tilde{y},\tilde{v},\tilde{u})
\end{align*}
Now in total, we end up with a much more tractible relaxed problem:

\begin{align*}\text{Min. } &\tilde{\mathcal{F}}(\tilde{t},\tilde{y},\tilde{v},\tilde{u}):= \int_0^S ((\tilde{t}(s)-1)^2+1)\abs{\tilde{u}(s)} + \abs{\tilde{y}_2(s)}^2 \tilde{v}(s) ds +(\tilde{y}_1(S)-1)^2\\
\text{where }& \tilde{t}' = \tilde{v},\, \tilde{t}(0) = 0, \tilde{t}(S) = T\\
  &\tilde{y}'_1 = \abs{\tilde{u}},\, \tilde{y}_1(0) = 0\\
  &\tilde{y}'_2 = \tilde{u}, \, \tilde{y}_2(0) = 0
\end{align*}
and $\tilde{v},\tilde{u} \in L^\infty([0,T]), 0\leq \tilde{v} \leq 1, \abs{\tilde{u}} \leq 1, y \in W^{1,\infty}([0,T];\R^2)$ for a free $S>0$.

As one of its minimizers this has
\[(\tilde{v},\tilde{u})(s) = \begin{cases}
                      (1,0) & \text{ for } s \in [0,1[\\
                      (0,1) & \text{ for } s \in [1,5/4[\\
                      (0,-1) & \text{ for } s \in [5/4,3/2[\\
                      (1,0) & \text{ for } s \in [3/2,T+1/2[\\
                     \end{cases}\]
resulting in our curve consisting of line segments.

\subsection{Outline}

The rest of this paper is devoted to making the above process rigorous on a general class of problems. In \autoref{sec:verticalODE}, we will define space-time curves and the associated space-time-relaxed ODEs. In \autoref{sec:Energy} we will discuss the appropriate relaxed energies and show that the resulting relaxed problem has a minimizer in case the original problem featured convexity. In \autoref{sec:Young} we will show that in the non-convex case, the relaxed problem still has a solution in terms of a $L^\infty$-Young measure and we will discuss the relation between this solution and the DiPerna-Majda measure in \cite{KruzikRoubicek99}. Finally in \autoref{sec:Banach}, we will show that this relaxation procedure can be easily extended to Banach spaces.

\subsection{Acknowledgements}

The author would like to especially thank Martin Kružík for pointing him toward the problem and for providing helpful comments on the manuscript. The author also acknowledges the support of the Primus research 
programm PRIMUS/19/SCI/01 and the University Centre UNCE/SCI/023 of 
Charles University. Moreover he thanks for the support of the program 
GJ17-01694Y of the Czech national grant agency (GA\v{C}R).

\subsection{Notation}

In the following, $t$ will be the original time variable while $s$ will be a new parameter along the space-time curve. Since many quantities occur in versions depending on either of them, for any time dependent quantity $a$, the derivative in $t$ will be denoted by $\dot{a} := \frac{da}{dt}$, while the corresponding $s$-dependent version will be denoted by $\tilde{a}$ and its derivative in $s$ will be written as $\tilde{a}' := \frac{d\tilde{a}}{ds}$.

We will use the usual notation $W^{1,p}(\Omega;\R^k)$ for Sobolev spaces and in general understand differential equations to be true in that sense only. We will use $M(I;\R^k)$ to denote the space of $\R^k$-valued Radon measures on $I$. The notation used for Young-Measures will be explained in \autoref{sec:Young}.

\subsection{The general problem and regularity conditions} \label{subsec:generalProblem}

The general non-relaxed problem will be given by
\begin{align*}
 \text{Minimize } & \mathcal{F}(y,u) := \int_0^T f(t,y(t),u(t)) dt + g(y(T))\\
 \text{where } & \dot{y}(t) = A(t,y(t),u(t)) \quad \forall t\in [0,T]\\
 &y(0) = y_0
\end{align*}
where $y:[0,T]\to \R^n$, $u:[0,T] \to \R^k$ as well as $f:[0,T] \times \R^n \times \R^k \to \R\cup\{+\infty\}$, $g:\R^n \to \R\cup\{+\infty\}$ and $A: [0,T] \times \R^n \times \R^k \to \R^n$. We note at this point that it is possible to restrict $y$ and $u$ to subsets of $\R^n$ and $\R^k$ respectively, but we will refrain from doing so.\footnote{Restriction of $y$ to a closed set $\overline{\Omega}$ can be equally achieved by setting $f$ to $+\infty$ at outside points, but $y$ will turn out to be uniformly bounded for minimizers anyway (See \autoref{lem:uniformBounds}). Restriction of $u$ to a compact set is not really interesting as we specifically interested in studying the case $\abs{u(t)} \to \infty$. It might be of interest to restrict in certain directions though, which can easiest be achieved by splitting the control variables into unbounded variables, which are transformed and bounded variables, which are kept. For sake of readability we will however not provide details on this.}

We will require a certain amount of regularity to show existence in the problem. Abstractly speaking we need the following properties.

\begin{enumerate}
 \item Existence of the relaxed energy and ODE, specifically $\tilde{f}$ and $\tilde{A}$ (see Definitions \ref{def:relaxedODE} and \ref{def:relaxedEnergy})
 \item Existence of unique solutions to the relaxed ODE, depending continuously on the controls
 \item (Lower semi-)continuity of the energy under convergence
\end{enumerate}

We achieve this by taking the following assumptions throughout the paper:

\begin{definition}[General assumptions] \label{def:assumptions} Let $A,f,g$ as above. Then we will generally assume:
\begin{enumerate}
 \item For all $t\in [0,T]$, $y\in\R^n$ and $u\in \R^k$ with $\abs{u} = 1$, the limits
 \[\lim_{\lambda\to\infty} \frac{f(t,y,\lambda u)}{\lambda} \text{ and } \lim_{\lambda\to\infty} \frac{A(t,y,\lambda u)}{\lambda}\]
 exist.
 \item The right hand side $A$ is continuous in $t$ and there exists a uniform Lipschitz constant $L\in \R$ such that for all $t\in [0,T]$ and $u\in \R^k$ we have
 \[\abs{A(t,y_1,u)-A(t,y_2,u)} \leq L \abs{y_1-y_2} \quad \forall y_1,y_2\in\R^n.\]
 Furthermore we require a constant $C>0$ such that for all $t\in [0,T]$, $y\in\R^n$ and $u\in \R^k$ we have
 \[\abs{A(t,y,u)} \leq C(\abs{u}+1).\]
 \item The integrand $f$ is equicontinuous in $t$ and $y$. Additionally there exists a constant $c> 0$ such that $t\in [0,T]$, $y\in\R^n$ and $u\in \R^k$ we have 
 \[\abs{f(t,y,u)} \geq c(\abs{u}-1) \]
 Furthermore $g$ is a continuous function bounded from below. 
 \item[3'.] Only in the convex case (\autoref{subsec:convex}): For all $t\in [0,T]$, $y\in\R^n$, the function $f(t,y,.)$ is convex and has a minimizer at $0$.\footnote{The last condition is mostly technical. There needs to be a minimum for $f(t,y,.)$ due to coercivity and similarly the set of minima should stay bounded for solutions with finite energy. We can thus replace all occurences of $u$ with a shifted version, having a minimum in zero.} 
\end{enumerate}
\end{definition}

It should be remarked that those assumptions are by far not the most general possible. Instead they have been chosen to aid in understanding the space-time relaxation while still allowing for a wide range of problems to be covered. As the arguments used are comparatively short, we invite the reader to simply modify them in the approriate places, should their problem not fit into the above assumptions.

\section{Space-time relaxed ODEs}\label{sec:verticalODE}

In this section we will cover only the ODE part of the problem, specifically
\begin{align*}
\dot{y}(s) = A(t,y(t),u(t)) \,\, \forall t\in [0,T].
\end{align*}
which we transform into the following space-time relaxation:

\begin{definition}[Space-time Relaxed ODE] \label{def:relaxedODE}
 The relaxed version of the ODE is given by
\begin{align*}
 & \tilde{t}'(s) = v(s) \\
 & \tilde{y}'(s) = \tilde{A}(\tilde{t}(s),\tilde{y}(s),v(s),\tilde{u}(s))\\
 &\tilde{y}(0) = y_0, \tilde{t}(0) = 0, \tilde{t}(S) =T
\end{align*} 
where 
\[\tilde{A}(\tilde{t},\tilde{y},\tilde{v},\tilde{u}) := \tilde{v}  A(\tilde{t},\tilde{y},\tilde{u}/\tilde{v})\]
for $\tilde{v}>0$ and we extend to $\tilde{v}=0$ by
\[\tilde{A}(\tilde{t},\tilde{y},0,\tilde{u}) := \lim_{\tilde{v}\searrow 0}  \tilde{v} A(t,\tilde{y},\tilde{u}/\tilde{v})\]
We will generally understand $(\tilde{t},\tilde{y})$ to be a solution to the ODE in a $W^{1,\infty}$ sense, i.e. the right hand side will be an $L^\infty$-function in $s$ and correspond to the weak derivative of $\tilde{t}$ and $\tilde{y}$ respectively.
\end{definition}

At this point, it should be noted that $\tilde{A}$ corresponds to an extension of $\tilde{v}A$ onto the compactification of $\R^k$ in the style of the projective plane. Specifically we identify points $\tilde{u}/\tilde{v} \in \R^k$ with tuples $(\tilde{v},\tilde{u}) \in \R^+ \times \R^k$, which leaves tuples $(0,\tilde{u})$ to represent a point at infinity in the direction of $\tilde{u}$. As this representation is not unique, we expect some sort of homogeneity:

\begin{lemma}[Homogeneity of $\tilde{A}$]\label{lem:A1homog}
 The relaxed right hand side $\tilde{A}$ as given in the previous definition is $1$-homogeneous in $(\tilde{v},\tilde{u})$. That is for any $\tilde{t},\tilde{y},\tilde{v},\tilde{u}$ and $\lambda > 0$ we have
 \[\tilde{A}(\tilde{t},\tilde{y},\lambda \tilde{v},\lambda\tilde{u}) =\lambda \tilde{A}(\tilde{t},\tilde{y},\tilde{v},\tilde{u}).\]
\end{lemma}

\begin{proof}
 If $\tilde{v} > 0$ then per definition
 \[\tilde{A}(\tilde{t},\tilde{y},\lambda \tilde{v},\lambda\tilde{u}) = \lambda \tilde{v} A(\tilde{t},\tilde{y},(\lambda\tilde{u})/(\lambda \tilde{v})) = \lambda \tilde{A}(\tilde{t},\tilde{y},\tilde{v},\tilde{u}).\]
 But then $\tilde{A}$ is the extension of a $1$-homogeneous function and thus itself $1$-homogeneous, as
 \[\lambda \tilde{A}(\tilde{t},\tilde{y},0,\tilde{u}) = \lambda \lim_{\tilde{v}\searrow 0} \tilde{v} A(\tilde{t},\tilde{y},\tilde{u}/\tilde{v}) = \lim_{\tilde{v}\searrow 0} (\lambda \tilde{v}) A(\tilde{t},\tilde{y},(\lambda\tilde{u})/(\lambda \tilde{v})) = \tilde{A}(\tilde{t},\tilde{y},0,\lambda \tilde{u}) \qedhere\]
\end{proof}

\begin{definition}[Trace of a solution]
 Let $\tilde{t},\tilde{y},\tilde{u}$ be a solution of the relaxed ODE. Then the trace of this solution is given by the set
 \[\left\{(\tilde{t}(s),\tilde{y}(s))\middle| s\in [0,S]\right\} \subset [0,T] \times \R^n\]
\end{definition}

In this context, the name trace is to be understood in the meaning of the trace of a curve. In fact the next lemma shows that we can reparameterize our solutions without invalidating the ODE or changing the trace.

\begin{lemma}[Invariance under reparametrization]
If $(\tilde{t},\tilde{y},\tilde{v},\tilde{u})$ is a solution of the relaxed ODE and $\phi:[0,\hat{S}] \to [0,S]$ is a surjective, increasing Lipschitz-function, then \[(\tilde{t}\circ \phi,\tilde{y} \circ \phi, \tilde{v} \circ \phi \phi', \tilde{u} \circ \phi \phi')\] is another solution of the relaxed ODE with the same trace.
\end{lemma}

\begin{proof} Both solutions share the same trace as $(\tilde{t}\circ \phi,\tilde{y} \circ \phi)$ is just a reparametrization of the underlying curve. Next let us show that the reparametrization also solves the ODE. For almost all $\hat{s}\in[0,S]$ we have
\[\frac{d (\tilde{t} \circ \phi) }{d\hat{s}}(\hat{s}) = \frac{d\tilde{t}}{ds}(\phi(\hat{s})) \frac{d\phi}{d\hat{s}} = \tilde{v} \circ \phi (\hat{s}) \phi'(\hat{s})\]
as well as using the $1$-homogeneity of $\tilde{A}$
\begin{align*}
 \frac{d}{d\hat{s}} \tilde{y} \circ \phi (\hat{s}) &= \phi'(\hat{s}) \tilde{y}'(\phi(\hat{s})) = \phi'(\hat{s}) \tilde{A}(\tilde{t}\circ \phi(\hat{s}),\tilde{y} \circ \phi(\hat{s}), \tilde{v} \circ \phi(\hat{s}), \tilde{u} \circ \phi (\hat{s}))\\
 &= \tilde{A}(\tilde{t}\circ \phi(\hat{s}),\tilde{y} \circ \phi(\hat{s}), \phi'(\hat{s}) \tilde{v} \circ \phi(\hat{s}), \phi'(\hat{s})\tilde{u} \circ \phi (\hat{s})) \qedhere
\end{align*}
\end{proof}

\begin{remark}
 The reparametrization does not have to be strictly increasing. In a sense, we can ``wait'' a certain amount of time.
\end{remark}

\begin{definition}[Associated solutions]
Let $y,u$ be a solution of the original problem and $\tilde{t},\tilde{y},\tilde{v},\tilde{u}$ a solution  of the relaxed problem. We call those solutions associated, if 
\[\tilde{y}(s) = y(\tilde{t}(s))\text{ and } u(\tilde{t}(s)) \tilde{v}(s) = \tilde{u}(s) \text{ for almost all } s \in [0,S].\]
\end{definition}

\begin{lemma}[Existence of associated relaxed solutions]\label{lem:associatedODEsolution}
For any solution $y,u$ of the original problem there exists an associated relaxed solution of the form
\begin{align*}
 \tilde{t}(s) = s, \tilde{v}(s) = 1, \tilde{y}(s) = y(s), \tilde{u}(s) = u(s)
\end{align*}
The trace of this solution is exactly the graph of $y$.
\end{lemma}

\begin{proof}
 Follows directly from the definition.
\end{proof}

\begin{remark}
 The obverse is not true, otherwise we would not need a relaxation. In fact the relaxed solutions without a classical equivalent are those for which the graph of $y$ contains vertical parts, as seen in the introduction.
\end{remark}

\begin{definition}[Normalized solutions]
 A solution $\tilde{t},\tilde{y},\tilde{v},\tilde{u}$ is called normalized if for almost all $s\in [0,S]$ we have $\max(\tilde{v},\abs{\tilde{u}}) = 1$.
\end{definition}

\begin{lemma}[Uniqueness of normalization] \label{lem:uniqNormalization}
 Let $(\tilde{t},\tilde{y},\tilde{v},\tilde{u})$ be a solution to the relaxed problem. Then there is exactly one normalized solution that can be reparameterized to $(\tilde{t},\tilde{y},\tilde{v},\tilde{u})$.
\end{lemma}

\begin{proof}
 The proof is essentially parametrization of a curve by arc-length. Specifically define
 \[l: [0,S] \to \R; \hat{s}\mapsto \int_0^{\hat{s}} \max(\tilde{v}(s),\abs{\tilde{u}(s)}) ds.\]
 Then $l$ is continuous and monotone and there exists a right-inverse $\phi:[0,\hat{S}]$ with $l \circ \phi = id$. We cannot reparameterize using $\phi$ directly, as $\phi$ can have jumps, but we note that if $\phi(x^-) < \phi(x^+)$ for an $x\in[0,\hat{S}]$, then $v = 0$ and $\tilde{u} = 0$ almost everywhere on $[\phi(x^-),\phi(x^+)]$. But since $\tilde{A}$ is 1-homogeneous, this means that $\tilde{t}$ and $\tilde{y}$ are constant on that interval. As a result an easy calculation shows that
 \[(\tilde{t}\circ \phi,\tilde{y} \circ \phi, \tilde{v} \circ \phi \phi', \tilde{u} \circ \phi \phi')\]
 is defined almost everywhere and also a solution of the problem. Furthermore
 \[ 1 = (l \circ \phi)' = \max(\tilde{v} \circ \phi, \tilde{u} \circ \phi) \phi'\]
 almost everywhere, and thus the new solution is normalized. Finally reparameterizing the new solution using $l$ as a parametrisation yields the original solution $(\tilde{t},\tilde{y},\tilde{v},\tilde{u})$.
\end{proof}

\begin{remark}
 The choice of normalization is somewhat arbitrary. As we have seen, solutions can be reparameterized using arbitrary monotone changes of variables. However there seems to be no useful canonical way to do so. 
 
 Looking at \autoref{lem:associatedODEsolution}, requiring $\tilde{v}(s)=1$ seems to be most natural, but of course this is only possible if there are no vertical parts. As soon as $\tilde{v}(s)=0$ for more than a neglible set of $s\in[0,S]$, this is no longer possible. Furthermore this will not give us an $L^\infty$ bound on $\tilde{v}$ and $\tilde{u}$, which we want to use for convergence later.

 An interesting object is the space-time curve $(\tilde{t},\tilde{y})$. While setting $\tilde{v}=1$ corresponds to parametrizing this curve as a graph of $t$, another possibility would be parametrization by arc-length. It is possible to do so in this problem as well. However this would only give us an indirect control of $\tilde{u}$ via $\tilde{A}$ which is not ideal later on.
 
 One might use $\tilde{v}^2+\abs{\tilde{u}}^2=1$ instead of $\max(\tilde{v},\abs{u})=1$ but sticking to the latter allows us to treat $\tilde{v}$ and $\tilde{u}$ more independently.
\end{remark}

Next let us show that given reasonable $\tilde{u}$ and $\tilde{v}$ there always exists a unique solution to the ODE and that this solution depends continuously on $\tilde{u}$ and $\tilde{v}$.
\begin{proposition}[ODE solutions for fixed controls] \label{prop:exODE}
 Fix $S>0$ and let $\tilde{u}\in L^\infty([0,S];\R^k)$ and $\tilde{v}\in L^\infty([0,S];\R^+)$ with $\norm[1]{v} = T$. Then there exists a unique corresponding solution $\tilde{t},\tilde{y}$ to the space time ODE. Furthermore if $A$ is linear, finding the corresponding solution is a weak-$\star$ continuous map from $L^\infty$ to $W^{1,\infty}$.
\end{proposition}

\begin{proof}
 For $\tilde{t}$, we note that $\tilde{t}' = \tilde{v}$, $\tilde{t}(0) = 0$ already determines $\tilde{t}$ uniquely and the continuous dependence on $\tilde{v}$ is also clear.

 For $\tilde{y}$, using the Lipschitz condition in \autoref{def:assumptions} (2.), we get that
 \[\tilde{y}\mapsto \tilde{A}(\tilde{t},\tilde{y},\tilde{v}(s),\tilde{u}(s))\]
 is uniformly Lipschitz. Thus short term existence of a unique $\tilde{y}$ follows from standard ODE theory and existence on $[0,S]$ follows from iterating this argument as $\tilde{A}$ is bounded and $\tilde{y}$ is not restricted to a subset of $\R^n$.
 
 Now let $\tilde{v}_k,\tilde{u}_k \rightharpoonup^\star \tilde{v},\tilde{u}$ in $L^\infty$ and denote the corresponding solutions by $\tilde{t}_k,\tilde{y}_k$. We already know that $\tilde{t}_k \rightharpoonup \tilde{t}$ in $W^{1,2}$ and thus $\tilde{t}_k \to \tilde{t}$ uniformly. As we have
 \[\abs{\tilde{A}(\tilde{t},\tilde{y},\tilde{v},\tilde{u})} = \abs{\tilde{v}}\abs{A(\tilde{t},\tilde{y},\tilde{u}/\tilde{v})} \leq C \tilde{v}\left(\abs{\frac{\tilde{u}}{\tilde{v}}} +1 \right) = C(\abs{\tilde{u}}+1)\]
 which in particular extends to $\tilde{v} = 0$, we know that $\tilde{y}_k$ is uniformly bounded in $W^{1,\infty}$ and thus any subsequence has a weak-$\star$ converging subsequence (not relabeled) with a limit $\tilde{y}$. In particular we have $\tilde{y}_k\to\tilde{y}$ uniformly. Now let $\phi \in L^1([0,S];\R^n)$. Then
 \begin{align*}
  \phantom{{}={}} \int_0^S \phi \cdot \left(\tilde{y}_k - \tilde{A}(\tilde{t},\tilde{y},\tilde{v},\tilde{u}) \right)ds 
  = \int_0^S \phi \cdot \left(\tilde{A}(\tilde{t}_k,\tilde{y}_k,\tilde{v}_k,\tilde{u}_k) - \tilde{A}(\tilde{t},\tilde{y},\tilde{v}_k,\tilde{u}_k) \right)ds \\ + \int_0^S \phi \cdot \left(\tilde{A}(\tilde{t},\tilde{y},\tilde{v}_k,\tilde{u}_k) - \tilde{A}(\tilde{t},\tilde{y},\tilde{v},\tilde{u}) \right)ds
 \end{align*}
 Here the first integral converges to $0$ due to the Lipschitz-continuity of $\tilde{A}$ and the uniform convergence of $\tilde{y}_k$, while for the second integral we note that if $A$ is linear, then $\int_0^S \phi \cdot \tilde{A}(\tilde{t},\tilde{y},.,.) ds$ is a linear functional in $(L^\infty)^\star$ and so the integral also converges to $0$. Thus by uniqueness of limits, $\tilde{y}' = \tilde{A}(\tilde{t},\tilde{y},\tilde{v},\tilde{u})$ and since this limit is unique as the solution to the relaxed ODE, the original sequence converges as required.
\end{proof}

\begin{remark}[The $1$-homogeneous case]

The case in which $A$ scales linear in $u$ is somethat special, as it allows us to restate the relaxed problem in a similar linear form. In fact the important property here is homogeneity of degree 1, that is we need that $A(\lambda u) = \lambda A(u)$ for $\lambda \geq 0$, so cases such as $A(u) = \abs{u}$ from \autoref{ex:updown} are covered as well.

In this case
\[\tilde{A}(\tilde{t},\tilde{y},\tilde{v},\tilde{u}) = \tilde{v} A(\tilde{t},\tilde{y},\tilde{u}/\tilde{v}) = A(\tilde{t},\tilde{y},\tilde{u})\]
which immediately extends to $\tilde{v}=0$ and simplifies the relaxed ODE to
\begin{align*}
 \tilde{t}' &= \tilde{v}\\
 \tilde{y}' &= A(\tilde{t},\tilde{y},\tilde{u})
\end{align*}
where the right hand side in the second equation now is linear/$1$-homogeneous in $\tilde{u}$ as well as in $(\tilde{v},\tilde{u})$.
\end{remark}

\section{Relaxed energies: Concentrations and the convex case}\label{sec:Energy}

\subsection{The relaxed energy}

We will now apply a similar procedure to relax the energy functional.

\begin{definition}[Relaxed energy] \label{def:relaxedEnergy}
 We define the relaxed integrand by
 \[\tilde{f}(\tilde{t},\tilde{y},\tilde{v},\tilde{u}) := \tilde{v} f(\tilde{t},\tilde{y},\tilde{u}/\tilde{v})\]
 for $\tilde{v}> 0$ and by the limit $\tilde{v}\searrow 0$ for $\tilde{v} =0$. The relaxed energy is then given by
 \[\tilde{\mathcal{F}}(\tilde{t},\tilde{y},\tilde{v},\tilde{u}) := \int_0^S \tilde{f}(\tilde{t},\tilde{y},\tilde{v},\tilde{u}) ds + g(\tilde{t}(S)).\]
\end{definition}

Let us first assemble some of the more immediate properties of the relaxed integrand.

\begin{lemma}[Homogeneity of $\tilde{f}$]
 The relaxed integrand $\tilde{f}$ is $1$-homogeneous in $(\tilde{v},\tilde{u})$.
\end{lemma}

\begin{proof}
 The proof is identical to the one of \autoref{lem:A1homog}.
\end{proof}

\begin{lemma}[Reparametrization of integrands]  \label{lem:energyParamInvariance}
 The relaxed energy is invariant under orientation preserving reparametrization.
\end{lemma}

\begin{proof}
 Let $\phi: [0,\hat{S}] \to [0,S]$ be a piecewise continuously differentiable orientation preserving change of coordinates.
 The second term is trivial since $\tilde{t} \circ \phi (\hat{S}) = \tilde{t}(S)$. For the first term, we use the $1$-homogeneity of $\tilde{f}$ and simply perform a change of variables:
 \begin{align*}
  &\phantom{{}={}}\int_0^{\hat{S}} \tilde{f}(\tilde{t} \circ \phi, \tilde{y} \circ \phi, \tilde{v} \circ \phi \phi', \tilde{u} \circ \phi \phi') d\hat{s}\\
  &= \int_0^{\hat{S}} \tilde{f}(\tilde{t} \circ \phi, \tilde{y} \circ \phi, \tilde{v} \circ \phi , \tilde{u} \circ \phi ) \phi' d\hat{s} = \int_0^{S} \tilde{f}(\tilde{t} , \tilde{y} , \tilde{v}  , \tilde{u}  ) d\hat{s}\qedhere
 \end{align*}

\end{proof}

\begin{lemma}[Energy for associated solutions]
 Let $y,u$ be a solution to the original ODE and $\tilde{t},\tilde{y},\tilde{v},\tilde{u}$ an associated solution to the relaxed problem. Let $\mathcal{F}$ be an energy and $\tilde{\mathcal{F}}$ the corresponding relaxed energy. Then
 \[\mathcal{F}(y,u) = \tilde{\mathcal{F}}(\tilde{t},\tilde{y},\tilde{v},\tilde{u})\]
\end{lemma}

\begin{proof}
 According to \autoref{lem:energyParamInvariance}, we can use any parametrization, so we will use $\tilde{t}(s) = s, \tilde{v}=1,\tilde{y} = y,\tilde{u} = u$ as in \autoref{lem:associatedODEsolution}. Then
 \begin{align*}
  \tilde{\mathcal{F}}(\tilde{t},\tilde{y},\tilde{v},\tilde{u}) &= \int_0^S \tilde{f}(\tilde{t},\tilde{y},\tilde{v},\tilde{u})ds + g(\tilde{t}(S)) = \int_0^S \tilde{v} f(\tilde{t},\tilde{y},\tilde{u}/\tilde{v}) ds + g(\tilde{t}(S)) = \int_0^T f(s,y,u) ds + g(T) = \mathcal{F}(y,u). \qedhere
 \end{align*}
\end{proof}

The idea of the existence theory is classical, but the bounds on the original problem need to be translated into bounds for the relaxed problem. It turns out that this works quite well.

\begin{proposition}[Uniform bounds for almost minimizers] \label{lem:uniformBounds}
 Assume that conditions 1-3 from \autoref{def:assumptions} hold and fix $E_0>0$. Now let $(\tilde{t},\tilde{y},\tilde{v},\tilde{u})$ be a normalized solution to the relaxed problem on the interval $[0,S]$ with $\tilde{\mathcal{F}}(\tilde{t},\tilde{y},\tilde{v},\tilde{u}) < E_0$. Then $S \leq 2T + \frac{1}{c}E_0$ as well as
 \begin{align*}
  \norm[1]{\tilde{u}} &\leq \frac{1}{c}E_0 +T &
  \norm[1]{\tilde{v}} & = T\\
  \norm[\infty]{\tilde{t}} &\leq  2T + \frac{1}{c}E_0 & 
  \norm[\infty]{\tilde{y}} &\leq 4CT+ \frac{2C}{c}E_0+\abs{y_0} \\
  \norm[\infty]{\tilde{t}'} &\leq  1 & 
  \norm[\infty]{\tilde{y}'} &\leq 2C
 \end{align*}
 where $c$ is the constant of the lower bound on $f$ in condition 3 and $C$ is from the upper bound on $A$ in condition 2. Note that all those bounds only depend on $E_0$ and initial data of the problem.
\end{proposition}

\begin{proof}
 As the solution is normalized, we have $\max(\tilde{v}(s),\abs{\tilde{u}(s)}) = 1$ and most importantly $\norm[\infty]{\tilde{u}} \leq 1$. We also note that $ \norm[1]{\tilde{t}'} = \norm[1]{\tilde{v}} = T$ per definition is always bounded.

From the lower bound on $f$ in condition 3, we get that
\[\tilde{f}(\tilde{t},\tilde{y},\tilde{v},\tilde{u}) = \tilde{v} f(\tilde{t},\tilde{y},\tilde{u}/\tilde{v}) \geq \tilde{v} c(\abs{\tilde{u}}/\tilde{v}-1) \geq c\abs{\tilde{u}}-\tilde{v}\]
for $\tilde{v} >0$ and then by the continuous extension also for $\tilde{v}=0$. Integrating yields
\[\int_0^S \tilde{f}(\tilde{t},\tilde{y},\tilde{v},\tilde{u}) ds \geq c\int_0^S \abs{\tilde{u}} ds - \int_0^S \tilde{v} ds = c \norm[1]{\tilde{u}}-T.\]

So as a result we get that $\norm[1]{\tilde{u}}$ is uniformly bounded.

As a consequence, the set 
\[\{s\in[0,S]: \tilde{v}(s) < 1\} \subset \{s\in[0,S]:\abs{\tilde{u}(s)}=1\}\] is bounded in mass by some constant $\norm[1]{\tilde{u}}$. This allows us to bound $S$, as then
\[S = \abs{\{s\in[0,S]: \tilde{v}(s) < 1\}} + \abs{\{s\in[0,S]: \tilde{v}(s) = 1\}} \leq \norm[1]{\tilde{u}}+T \leq 2T +\frac{1}{c}E_0.\] 

Furthermore from the upper bound on $A$ we get
\[\abs{\tilde{A}(\tilde{t},\tilde{y},\tilde{v},\tilde{u})} = \abs{\tilde{v} A(\tilde{t},\tilde{y},\tilde{u}/\tilde{v})} \leq C(1+\abs{\tilde{u}}) \leq 2C\]
for $\tilde{v}>0$ and by continuity also for $\tilde{v}=0$. So since $\tilde{y}' = \tilde{A}(\tilde{t},\tilde{y},\tilde{v},\tilde{u})$, we also get that $\norm[\infty]{\tilde{y}'} \leq 2C$ and thus $\norm[\infty]{\tilde{y}}\leq 2CS+\abs{y_0}$, which results in the stated bound. The same holds for $\tilde{t}_k$, since $\abs{\tilde{t}_k'(s)} = \abs{\tilde{v}_k(s)} \leq 1$ and thus $\tilde{t}(s) \leq S$.
\end{proof}

\begin{remark}[On problems with variable end-time] \label{rem:openEndTime}
 So far we only looked at problems where the final time $T$ is fixed. Of course this excludes a large class of optimal control problems, where $T$ depends in some way on the curve $y$, usually in the form of requiring $y(T)$ to reach a certain point or a set of points. It is however not hard to extend everything to this case as well.
 
 A reasonable assumption on those kind of problems is that a bounded energy also results in a bound for final time $T$, otherwise there might be no solution as $T\to\infty$ along a minimizing sequence. And since we have shown that for any minimizing sequence of normalized solutions, $S_k$ is uniformly bounded in terms of energy and $T$, this is condition is indeed enough to show existence in this case as well.
\end{remark}

\subsection{Existence in the convex case}\label{subsec:convex}

\begin{proposition}[Convexity of integrand]
 Let $f$ be an integrand satisfying conditions 3 and 3' in \autoref{def:assumptions}. (In particular that means that $f$ is continuous in $y$ and convex in $u$ with a minimum for $u=0$.) Then the relaxed integrand $\tilde{f}$ is also continuous in $\tilde{y}$ and convex in $\tilde{v}$ and $\tilde{u}$ each.
\end{proposition}

\begin{proof}
 For fixed $\tilde{v}>0$, the integrand $\tilde{f}$ is just a rescaling of $f$, so we inherit the continuity in $y$ and the convexity in $\tilde{u}$. The convexity of $\tilde{u}$ for $\tilde{v}=0$ is also inherited directly:
 \begin{align*}
  \tilde{f}(\tilde{t},\tilde{y},0,\lambda \tilde{u}_1 + (1-\lambda) \tilde{u}_2) &= \lim_{\tilde{v}\to 0} \tilde{v} f\left(\tilde{t},\tilde{y},\frac{\lambda \tilde{u}_1 + (1-\lambda) \tilde{u}_2}{\tilde{v}}\right)\\
  &\leq \lim_{\tilde{v}\to 0} \lambda \tilde{v} f\left(\tilde{t},\tilde{y},\frac{\tilde{u}_1}{\tilde{v}}\right) + (1-\lambda) \tilde{v} f\left(\tilde{t},\tilde{y},\frac{\tilde{u}_2}{\tilde{v}}\right)\\
  &= \lambda \tilde{f}(\tilde{t},\tilde{y},0,\tilde{u}_1)+ (1-\lambda) \tilde{f}(\tilde{t},\tilde{y},0,\tilde{u}_2)
 \end{align*}
 
 This leaves the convexity in $\tilde{v}$. First let $\tilde{v}_1,\tilde{v}_2 >0$. Then
 \begin{align*}
  \tilde{f}(\tilde{t},\tilde{y},\lambda \tilde{v}_1 + (1-\lambda) \tilde{v}_2,\tilde{u}) &= (\lambda \tilde{v}_1 + (1-\lambda) \tilde{v}_2)f\left(\tilde{t},\tilde{y},\frac{1}{\lambda \tilde{v}_1 + (1-\lambda) \tilde{v}_2} \tilde{u}\right)\\
 &\leq \lambda \tilde{v}_1f\left(\tilde{t},\tilde{y},\frac{1}{\lambda \tilde{v}_1} \tilde{u}\right)+(1-\lambda)\tilde{v}_2 f\left(t,\tilde{y},\frac{1}{(1-\lambda) \tilde{v}_2} \tilde{u}\right) \\&= \lambda \tilde{f}(\tilde{t},\tilde{y},\tilde{v}_1,\tilde{u})+ (1-\lambda) \tilde{f}(\tilde{t},\tilde{y},\tilde{v}_2,\tilde{u})
 \end{align*}
 where we use that $l \mapsto f(\tilde{t},\tilde{y},l\tilde{u})$ is convex with a minimum in $0$ and thus monotone for $l\geq 0$. Finally the case $\tilde{v}_1 =0$ or $\tilde{v}_2=0$ follows by taking the limit.
\end{proof}

\begin{theorem}[Existence of solutions (convex case)] \label{thm:existConv}
 Consider the initial optimal control problem. Assume that the conditions 1-3 and 3' from \autoref{def:assumptions} hold. Then the relaxed problem has a solution.
\end{theorem}

\begin{proof}
 We pick a minimizing sequence $(\tilde{t}_k,\tilde{y}_k,\tilde{v}_k,\tilde{u}_k)$ to the relaxed problem. Since the problem is invariant under reparametrization, we can choose this sequence to be normalized. We then apply \autoref{lem:uniformBounds} to get uniform bounds on that sequence. 

In particular this gives us a uniform bound on $S_k$. Now by the $1$-homogeneity, we have $\tilde{f}(\tilde{t}_k,\tilde{y}_k,0,0)= 0$ and $\tilde{y}_k' = \tilde{A}(\tilde{t}_k,\tilde{y}_k,0,0) = 0$. So we can extend all solutions by $\tilde{v}_k(s)=0,\tilde{u}_k(s)=0$ for $s > S_k$ to the same interval $[0,S]$ where $S:= \sup_{k\in\N} S_k$, without changing any of the associated norms.

We now use those bounds to extract a subsequence (not relabeled) and find functions $\tilde{t},\tilde{y},\tilde{v},\tilde{u}$ for which
\begin{align*}
 \tilde{t}_k &\to \tilde{t} \quad \text{ uniformly} \\
 \tilde{y}_k &\to \tilde{y} \quad \text{ uniformly} \\
 \tilde{v}_k &\rightharpoonup^* \tilde{v} \quad \text{in $L^\infty$} \\
 \tilde{u}_k &\rightharpoonup^* \tilde{u} \quad \text{in $L^\infty$} 
\end{align*}

Now since $\tilde{f}$ is continuous in $\tilde{t}$ and $\tilde{y}$ and convex in $\tilde{v}$ and $\tilde{u}$, by Tonelli's theorem $\tilde{\mathcal{F}}$ is lower semicontinuous under those convergences and so
\[\tilde{\mathcal{F}}(\tilde{t},\tilde{y},\tilde{v},\tilde{u}) \leq \liminf_{k\to\infty} \tilde{\mathcal{F}}(\tilde{t}_k,\tilde{y}_k,\tilde{v}_k,\tilde{u}_k) = \inf \tilde{\mathcal{F}}.\]
Furthermore by continuity
\[\tilde{y}' = \tilde{A}(\tilde{t},\tilde{y},\tilde{v},\tilde{u})\]
and by uniform convergence $\tilde{y}(0) = 0$ and $t(S) = T$, so $(\tilde{t},\tilde{y},\tilde{v},\tilde{u})$ is indeed a minimizer.
\end{proof}

\begin{remark}[On uniqueness] \label{rem:uniqueness}
 In general we cannot expect the relaxed problem to have an unique minimizer. Even if the original integrand is strictly convex, the relaxed integrand will not be. Not only is it invariant under reparametrizations but even if we normalize all solutions, the convexity might also no longer be strict in the limit $\tilde{v}\to 0$. The obvious way to avoid this would be to require uniform convexity of the original integrand, however this would imply a superlinear growth for $u\to \infty$, which precludes exactly the kind of concentrations for which this relaxation was required in the first place.
\end{remark}

\section{The nonconvex case and Young's generalized curves} \label{sec:Young}

To a reader familiar with relaxation methods used in the calculus of variations, it should be clear how to use Young-measures to prove an existence result for the non-convex case. We simply apply the classic techniques to the space-time relaxed problem. We would however still like to discuss those methods in some more details, as there are interesting connections to Young's generalized curves \cite{youngGeneralizedCurvesExistence1937} which are firmly at the root of everything that is done under the name of Young-measures today.

\subsection{Young measures and Generalized curves} \label{subsec:genCurves}

In this section we will repeat the idea of $L^\infty$-Young measures as well as Young's original concept of generalized curves. We will slightly deviate from his original definition by not normalizing (see \autoref{rem:YoungArclength}), but this is mostly a choice of convenience to highlight some of the features of the problem as the underlying mathematics are fundamentally the same.

\begin{definition}[$L^\infty$-Young measure]
 We call a family $\mu_s$, $s\in [0,S]$ of $\R^k$-valued Radon measure an $L^\infty$-Young measure\footnote{Which along with using $\mu_s$ to designate the whole family is a common abuse of notation.}, if there is a sequence $(u_k)_k:[0,S] \to \R^k$ of equibounded $L^\infty$ functions such that $\delta_{u_k(s)} \rightharpoonup \mu_s$ in the sense of measures for almost all $s\in [0,S]$.
\end{definition}

\begin{remark}
 It follows directly from the definition that $\supp \mu_s$ is bounded in some circle $B_R(0)$ where $R = \sup_{k\in\N} \norm[\infty]{u_k}$. It is also not hard to see that $\mu_s(\R^n) = 1$ for almost all $s\in[0,S]$. Because of this, $\mu_s$ is often defined to be a probability measure. While in some applications this is the right intuition, one should keep in mind that a priori there is no probability theory involved in the definition. It is equaly possible to interpret a Young measure as a function taking multiple values at once, where $\mu_s$ just indicates ``how much'' certain values are taken. 
\end{remark}

The following technical lemma will be quite useful in the sequel.

\begin{lemma}[Convergence of integrals] \label{lem:YoungConvInt}
 Let $\mu_s$ be an $L^\infty$-Young measure generated by $(u_k)_k$ and let $f_k: \R^k \to \R$ be a sequence of bounded continuous functions converging uniformly to $f:\R^k\to\R$. Then
 \[\int_{s_0}^{s_1} f_k(u_k(s)) ds \to \int_{s_0}^{s_1} \int_{\R^k} f(u) \,d\mu_s(u) ds\]
 for all $s_0,s_1 \in [0,S]$.
\end{lemma}

\begin{proof}
Since $u_k$ and $\mu_s$ are bounded, we can assume $f_k$ and $f$ to have compact support. Then we can estimate
 \begin{align*}
  &\phantom{{}={}} \abs{\int_{s_0}^{s_1} f_k(u_k(s)) ds- \int_{s_0}^{s_1} \int_{\R^k} f(u) d\mu_s(u) ds} \\
  &\leq \abs{\int_{s_0}^{s_1} f_k(u_k(s)) - f(u_k(s)) ds}+ \abs{\int_{s_0}^{s_1} f(u_k(s)) ds- \int_{s_0}^{s_1} \int_{\R^k} f(u) d\mu_s(u) ds} \\
  &\leq (s_1-s_0)\norm[\infty]{f_k-f} + \abs{ \int_{s_0}^{s_1} \int_{\R^k} f(u) d(\delta_{u_k(s)} - \mu_s)(u) ds}\to 0 \qedhere
 \end{align*}
\end{proof}

\begin{definition}[Generalized curves]
 Let $A \subset \R^n$ and $S > 0$. We call a pair $\gamma:[0,S] \to A, \nu_s: [0,S] \to M(\R^n)$ a generalized curve in $A$, if the following conditions hold:
 \begin{enumerate} 
  \item $\gamma$ is a continuous function.
  \item $\nu_s(\R^n) = 1$ for almost all $s\in [0,S]$
  \item For any $s_0,s_1 \in [0,S]$ we have
  \[\int_{s_0}^{s_1} \int_{\R^n} v d\nu_s(v) ds = \gamma(s_1) - \gamma(s_0) \]
 \end{enumerate}
\end{definition}

In our case, thanks to the space-time relaxation, we will only have to deal with curves where $\gamma$ is Lipschitz and the $\nu_s$ are uniformly bounded. We can immediately show that the latter implies the former:

\begin{lemma}[Bounded generalized curves] \label{lem:genCurveLipsch}
 Let $(\gamma,\nu_s)$ be a generalised curve on $[0,S]$ such that the $\nu_s$ are uniformly bounded, i.e. $\supp \nu_s \subset B_R(0)$. Then $\gamma$ is Lipschitz with constant $L\leq R$ and for almost all $s\in [0,S]$ we have 
 \[\dot{\gamma}(s) = \int_{R^m} v d\nu_s(v).\]
\end{lemma}
\begin{proof}
Per definition we have for any $s_0,s_1 \in [0,S]$:
\begin{align*}
  \abs{\gamma(s_1) - \gamma(s_0)} &= \abs{\int_{s_0}^{s_1} \int_{\R^n} v d\nu_s(v)ds}\\
  \leq \int_{s_0}^{s_1} \int_{\R^n} \abs{v} d\nu_s(v)ds
  &\leq \int_{s_0}^{s_1} \int_{\R^n} R d\nu_s(v)ds = R \abs{s_1-s_0}
 \end{align*}
 so $\gamma$ is Lipschitz with the correct constant. But then $\dot{\gamma}$ exists a.e. and by condition 3 we see that $\dot{\gamma}$ is identical to $\int_{\R^n} v d\nu_s(v)$ when integrated over any time interval and thus also identical for almost all $s$.
\end{proof}

The converse is not true. As the only connection between $\gamma$ and $\nu_s$ is given by  mean, we can find a family $\mu_s$ with arbitrary large support but a mean of $0$. Then $(0,\mu_s)$ is a generalized curve. However a Lipschitz-curve always induces a generalized curve with bounded support.

\begin{lemma}[Associated generalized curve]
 If $\gamma: [0,S] \to \R^n$ is a $W^{1,\infty}$-curve then $(\gamma,\nu_s)$, where $\nu_s = \delta_{\dot{\gamma}(s)}$, is a generalized curve (called the associated generalized curve).
\end{lemma}

\begin{proof}
 Conditions 1 and 2 are trivially fulfilled. For condition 3, we note that
 \[\int_{s_0}^{s_1} \int_{\R^n} v d\nu_s(v) ds = \int_{s_0}^{s_1} \dot{\gamma} ds = \gamma(s_1) - \gamma(s_0).\qedhere\]
\end{proof}

\begin{lemma}[Reparametrization of generalized curves]
 Let $(\gamma,\nu)$ be a generalized curve and $\phi:[0,\hat{S}] \to [0,S]$ a Lipschitz-continuous, monotone change of coordinates. Then the reparametrization $(\hat{\gamma},\hat{\nu}_s)$ is also a generalised curve, where $\hat{\gamma} = \gamma \circ \phi$ and
 \[\hat{\nu}_s: A \mapsto \nu_{\phi(s)}\left(\phi'(s)^{-1} A\right). \]
\end{lemma}

\begin{proof}
 The new curve $\hat{\gamma}$ is a well defined continuous curve, so condition 1 is fulfilled. As $\phi$ is Lipschitz, it is almost everywhere differentiable. Thus $\hat{\nu}_s$ is well defined for almost all $s\in[0,\hat{S}]$ and for almost all $s\in[0,\hat{S}]$ we have 
 \[\hat{\nu}_s(\R^n) = \nu_{\phi(s)}\left(\phi'(s)^{-1} \R^n\right) = \nu_{\phi(s)}(\R^n) = 1\]
 showing condition 2. Finally condition 3 is just a change of variables:
 \begin{align*}
  \int_{s_0}^{s_1} \int_{\R^n} v d\hat{\nu}_s(v) ds &= \int_{s_0}^{s_1} \int_{\R^n} \phi'(s) v d\nu_{\phi(s)}(v) ds \\
  = \int_{\phi(s_0)}^{\phi(s_1)} \int_{\R^n} v d\nu_{t}(v) dt &= \gamma(\phi(s_1))-\gamma(\phi(s_0))\qedhere
 \end{align*}

\end{proof}

\begin{remark}[On arc-length parametrization] \label{rem:YoungArclength}
 One might consider a generalized curve $(\gamma,\nu_s)$ as parameterized by arc-length if $\supp\nu_s \subset \S^{N-1}$.
 It is always possible to reparameterize a generalized curve in such a way, using the ususal arguments, however we will not do so. In Young's original definition generalized curves are always parametrized this way..
 
 Arc-length parametrization of $(\gamma,\nu_s)$ does in general not correspond to arc-length parametrization of $\gamma$. In fact if $(\gamma,\nu_s)$ is parametrized by arc-length and $\gamma$ is differentiable, we will have for almost all $s\in [0,S]$
 \[\abs{\dot{\gamma}(s)} = \abs{\int_{\S^{N-1}} v d\nu_s(v) } \leq \int_{\S^{N-1}} \abs{v} d\nu_s(v) = 1 \]
 with equality if and only if $\nu_s$ is concentrated at a single point, or in other words $\gamma$ is only parameterized by arc-length if $(\gamma,\nu)$ is the associated generalized curve of $\gamma$.
 
 The usefulness of this definition is that the limit of any sequence of arc-length parameterized curves will again be parameterized by arc-length. It also forms a bridge to the original generalized curves of Young and the definition of a Varifold, of which our generalized curves can be considered a special oriented case.
\end{remark}

\begin{lemma}[Convergence of generalized curves]\label{lem:convGenCurves}
 Let $\left((\gamma_k,(\nu_k)_s)\right)_{k\in\N}$ be a sequence of generalized curves such that $\supp (\nu_k)_s$ is uniformly bounded, $\gamma_k \to \gamma$ pointwise and $(\nu_k) \rightharpoonup \nu_s$ in the sense of measures. Then $(\gamma,\nu_s)$ is also a generalized curve.
\end{lemma}

\begin{proof}
 Let us verify the three conditions. Condition 2 follows directly from the definition of convergence of measures.
 
 Since $\supp (\nu_k)_s \subset B_R(\R^n)$ for some $R > 0$, \autoref{lem:genCurveLipsch} tells us that $\gamma_k$ is uniformly Lipschitz continuous. But then so is the limit $\gamma$, which implies condition 1.

 Finally looking at condition 3, we consider
 \[\int_{s_0}^{s_1} \int_{\R^n} v d(\nu_k)_s(v)ds = \gamma_k(s_1) - \gamma_k(s_0).\]
 As $\gamma_k \to \gamma$ pointwise, the right hand side converges to $\gamma(s_1)-\gamma(s_0)$. For the left hand side, we again use the bounded support and replace the integrand $v$ with a compactly supported function $f \in C_c^0(\R \times \R^n)$ such that $f(v) = v$ in $[0,S] \times B_R(\R^n)$. Then we have by convergence of the measures
 \begin{align*}\int_{s_0}^{s_1} \int_{\R^n}  v d(\nu_k)_s(v)ds&= \int_{s_0}^{s_1} \int_{\R^n}  f(v) d(\nu_k)_s(v)ds\\
  \to \int_{s_0}^{s_1} \int_{\R^n}  f(v) d\nu_s(v)ds &= \int_{s_0}^{s_1} \int_{\R^n}  v d\nu_s(v) ds
 \end{align*}
 and thus condition 3.
\end{proof}

\subsection{Existence in the fully relaxed problem}

We can now combine the relaxation of Sections \ref{sec:verticalODE} and \ref{sec:Energy} with the Young-measures of \autoref{subsec:genCurves} into a single approach which will cover the general case of $L^1$ optimal control problems. 

\begin{definition}[Fully relaxed problem]\label{def:fullyRelaxed}
Let $A, f$ and $g$ as before and define $\tilde{A}$ and $\tilde{f}$ according to Definitions \ref{def:relaxedODE} and \ref{def:relaxedEnergy}. Then the fully relaxed problem is given by
 \begin{align*}
 \text{Minimize } & \tilde{\mathcal{F}}(\tilde{t},\tilde{y},\mu_s) := \int_0^T \int_{\R^k} \tilde{f}(\tilde{t}(s),\tilde{y}(s),\tilde{v},\tilde{u}) \,d\mu_s(\tilde{v},\tilde{u}) ds + g(\tilde{y}(s))\\
 \text{where } & \tilde{y}'(s) = \int_{\R^{k+1}} \tilde{A}(\tilde{t}(s),\tilde{y}(s),\tilde{v},\tilde{u}) \,d\mu_s(\tilde{v},\tilde{u}) \quad \forall s\in [0,S]\\
 &\tilde{t}'(s) = \int_{\R^{k+1}} \tilde{v} \, d\mu_s(\tilde{v},\tilde{u})  \quad \forall s\in[0,S]\\
 &\tilde{y}(0) = y_0, \tilde{t}(0) = 0, \tilde{t}(S) = T
\end{align*}
where $\mu_s$ is a $L^\infty$-Young measure on the bounded set $[0,1]\times B_1(0) \subset \R^{k+1}$.
\end{definition}

\begin{remark}[$(\tilde{t},\tilde{y})$ as a generalized curve]
 The above definition of the fully relaxed problem is written in the way commonly used. There is however a different, equivalent way to look at it. If we fix $s \in [0,S]$, then we can also push forward $\mu_s$ using $\tilde{A}$ to get a measure $\nu_s= (\pi_{\tilde{v}},\tilde{A}(\tilde{t}(s),\tilde{y}(s),.,.))_* \mu_s$ or to be precise:
 \[\nu_s : U \mapsto \mu_s\left(\left\{(\tilde{v},\tilde{u}) \in \R^{k+1}: (\tilde{v},\tilde{A}(\tilde{t}(s),\tilde{y}(s),\tilde{v},\tilde{u})) \in U \right\}\right) \]
 
 Using this together with the definition of $\tilde{y}$, it follows that $((\tilde{t},\tilde{y}),\nu_s)$ is a generalized curve. Furthermore if $\mu_s$ is generated by $(\tilde{v}_k,\tilde{u}_k)_k$ and $\tilde{A}$, then by continuity of the pushforward, $((\tilde{t},\tilde{y}),\nu_s)$ is generated by the sequence of $(\tilde{t}_k,\tilde{y}_k)$ which solve the space-time ODE for $(\tilde{v}_k,\tilde{u}_k)$. This gives us an alternative formulation of the fully relaxed problem
 
 \begin{align*}
 \text{Minimize } & \tilde{\mathcal{F}}(\tilde{t},\tilde{y},\mu_s) := \int_0^T \int_{\R^k} \tilde{f}(\tilde{t}(s),\tilde{y}(s),\tilde{v},\tilde{u}) \,d\mu_s(\tilde{v},\tilde{u}) ds + g(\tilde{y}(s))\\
 \text{where } & \nu_s = (\pi_{\tilde{v}}, \tilde{A}(\tilde{t}(s),\tilde{y}(s),.,.) ) \quad \forall s\in [0,S]\\
 &\tilde{y}(0) = y_0, \tilde{t}(0) = 0, \tilde{t}(S) = T
\end{align*}
where $\mu_s$ is a $L^\infty$-Young measure on the bounded set $[0,1]\times B_1(0) \subset \R^{k+1}$ and $((\tilde{t},\tilde{y}),\nu_s)$ is a generalized curve uniquely defined through $\mu_s$.
\end{remark}

\begin{proposition}[Existence of a recovery sequence]
 Any solution $(\tilde{t},\tilde{y},\mu_s)$ to the fully relaxed problem is a limit of solutions $(\tilde{t}_k,\tilde{y}_k,(\tilde{v}_k,\tilde{u}_k))$ to the space-time relaxed problem in the sense that $\mu_s$ is generated by $(\tilde{v}_k,\tilde{u}_k)_k$ and:
 \begin{align*}
  \tilde{t}_k & \to \tilde{t} \text{ uniformly}\\
  \tilde{y}_k & \to \tilde{y} \text{ uniformly} \\
  \tilde{\mathcal{F}}(\tilde{t},\tilde{y},\mu_s) &= \lim_{k\to\infty} \tilde{\mathcal{F}}(\tilde{t}_k,\tilde{y}_k,\tilde{v}_k,\tilde{u}_k)
 \end{align*}
\end{proposition}

\begin{proof}
 Pick a generating sequence $(\tilde{v}_k,\tilde{u}_k)_k$ for the Young measure $\mu_s$ and let $\tilde{t}_k,\tilde{y}_k$ be the solutions to the corresponding space-time ODE. As the $(\tilde{v}_k,\tilde{u}_k)_k$ are uniformly bounded, \autoref{def:assumptions} and $\tilde{t}_k' = \tilde{v}_k$ give us uniform Lipschitz bounds on $\tilde{t}_k$ and $\tilde{y}_k$. Using those we can use the Arzela-Ascoli theorem to pick a subsequence (not relabeled) for which they converge uniformly.
 
 Finally $\tilde{f}(\tilde{t}_k,\tilde{y}_k,.,.)$ converges uniformly to $\tilde{f}(\tilde{t},\tilde{y},.,.)$ so convergence of the energy follows from \autoref{lem:YoungConvInt} and continuity of $g$.
\end{proof}

\begin{theorem}[Existence of fully relaxed solutions]
 Let $A$ satisfy conditions 1-3 from \autoref{def:assumptions}.
 Then the fully relaxed problem from \autoref{def:fullyRelaxed} has a minimizer.
\end{theorem}

\begin{proof}
 Take a minimizing sequence $(\tilde{t}_k,\tilde{y}_k,\tilde{v}_k,\tilde{u}_k)$ to the space-time relaxed problem. As before, we can assume $(\tilde{t}_k,\tilde{y}_k)$ to be normalized. Then by the same arguments as in the convex case in \autoref{thm:existConv}, we use \autoref{lem:uniformBounds} to obtain uniform bounds and extend all functions to a common interval $[0,S]$. Now since $\tilde{t}, \tilde{y}$ are bounded in $W^{1,\infty}$ and $\tilde{v},\tilde{u}$ are naturally bounded in $L^\infty$, we get a converging subsequence (not relabeled) and limits $S$, $(\tilde{t},\tilde{y}) ,\nu_s$ and $\mu_s$ such that
 \begin{align*}
  S_k & \to S \\
  (\tilde{t}_k,\tilde{y}_k) &\to (\tilde{t},\tilde{y}) \text{ uniformly}\\
  \delta_{(\tilde{t}_k(s),\tilde{y}_k(s))} &\rightharpoonup \nu_s \text{ in the sense of measures}\\
  \delta_{(v_k,\tilde{u}_k)(s)}& \rightharpoonup \mu_s \text{ as a Young measure.}
 \end{align*}
 Note that by \autoref{lem:convGenCurves}, the limit $((\tilde{t},\tilde{y}) ,\nu_s)$ is also a generalized curve. Furthermore since $\tilde{A}$ is uniformly continuous in $\tilde{t}$ and $\tilde{y}$, the limit still solves the ODE
 \[ \nu_s= (\pi_{\tilde{v}},\tilde{A}(\tilde{t}(s),\tilde{y}(s),.,.))_* \mu_s.\]
 So in total $(t,\tilde{y}, \mu_s)$ is a an admissible solution to the relaxed problem and by the same arguments as in the last proposition we have
 \[\tilde{\mathcal{F}}(t,\tilde{y}, \mu_s) = \liminf_{k\to \infty} \tilde{\mathcal{F}}(t_k,\tilde{y}_k, \delta_{(v_k,\tilde{u}_k)(.)})\]
 and thus it is also a minimizer.
\end{proof}

\begin{remark}[On generalisations and classical solutions]
 All our previous considerations from \autoref{rem:openEndTime} and \autoref{rem:uniqueness} apply here as well.
 
 Furthermore, as expected a minimizer of the fully relaxed problem defaults back onto a minimizer of the space-time relaxed problem if the Young measure corresponds to a function, i.e. $\mu_s = \delta_{\tilde{v}(s),\tilde{u}(s)}$ for some $\tilde{v},\tilde{u}$. 
 
 The other direction, removing the space-time relaxation is less obvious. If we want a Young measure solution to the original problem, requiring that $\tilde{t}$ is strictly monotonous, is not enough. What we additionally need is that if $(\tilde{v}_k,\tilde{u}_k)_k$ generates $\mu_s$, then $\left(\frac{\tilde{u}_k}{\tilde{v}_k}\right)_k$ generates an $L^p$-Young measure for some $p\in[1,\infty]$, a condition that is much harder to enforce.
\end{remark}

\subsection{Relation to DiPerna-Majda measures}

As stated in the introduction, the original motivation was to provide a simpler alternative to the relaxation used in \cite{KruzikRoubicek99}. To close this section, let us now compare the two alternatives directly.

First let us compare space-time curves as a relaxation to the one used in \cite{KruzikRoubicek99}, where they used the space $W^1_\mu$ to generalize the curve $y$.

\begin{proposition}[Relation to $W^1_\mu$] \label{prop:W1mu}
 Consider the following space introduced by Souček  \cite{soucekSpacesFunctionsDomain1972}
 \begin{align*}
  W^1_\mu([0,T];\R^n) := & \left\{ (y,\dot{y}) \in L^1([0,T];\R^n) \times M([0,T];\R^n) ; \right. \\
  & \left. \exists (y_k)_{k\in\N} \subset W^{1,1}([0,T];\R^n): y_k \to y \text{ in } L^1, \dot{y}_k \rightharpoonup^* \dot{y} \text{ in }M \right\}
 \end{align*}
 Let $\tilde{t},\tilde{y}$ be a relaxed curve.\footnote{We ignore the right hand side for now and assume that $\tilde{t}'$ and $\tilde{y}'$ can be chosen arbitrary, i.e. assume that $\tilde{A}(\tilde{t},\tilde{y},.,.)$ is always invertible. If this is not the case, only a restricted set of relaxed curves, both in the space-time approach as well as in the $W^1_\mu$-sense, can actually be reached. Those restricted sets can still be related to each other by the same approach (compare the proof of \cite[Prop. 3.1, 3.2]{KruzikRoubicek99} with the proof of \autoref{prop:W1mu} here), but doing so is a bit more involved.} Then there exists a unique corresponding $(y,\dot{y}) \in W^1_\mu([0,T];\R^n)$ such that
 \begin{align*}
  y(t) &= \tilde{y}(\tilde{t}^{-1}(t)) \text{ whenever $\tilde{t}^{-1}$ is unique}\\
  \int_0^T \phi d\dot{y} &= \int_0^S \tilde{y}'(s) \phi(\tilde{t}(s)) ds.
 \end{align*}
 Conversely, if $(y,\dot{y}) \in W^1_\mu([0,T];\R^n)$, then there is at least one relaxed curve $\tilde{t},\tilde{y}$ corresponding to $(y,\dot{y})$ by the above process.
\end{proposition}

\begin{proof}
 As $\tilde{t}$ is monotone, whenever there are $s_0 < s_1$ such that $\tilde{t}(s_0) = \tilde{t}(s_1)$ we have to have $\tilde{t}(s) = \tilde{t}(s_0)$ for all $s\in [s_0,s_1]$. But then $|\{s\in[0,S]:\tilde{t}(s)=\tilde{t}(s_0)\}| > 0$ so there can only be countably many such intervals. Thus $\tilde{t}^{-1}$ is unique almost everywhere and so $y$ is defined almost everywhere on $[0,T]$ and thus well defined. The derivative $\dot{y}$ is already uniquely defined by the condition.
 
 Now define
 \[\tilde{t}_k(s) : [0,S] \to [0,T], T\frac{\tilde{t}(s) + \frac{s}{k}}{T+S/k} .\]
 Then $\tilde{t}_k \to \tilde{t}$ uniformly. Furthermore $\tilde{t}_k' >0$ so there exists a function $y_k:[0,T] \to \R^n$ corresponding to $s\mapsto (\tilde{t}_k,\tilde{y})$. Now by change of variables and since $\tilde{y}$ and $\phi$ are Lipschitz
 \begin{align*}
  \int_0^T \dot{y}_k \phi dt = \int_0^S \tilde{y}'(s) \phi(\tilde{t}_k(s)) ds \to \int_0^S \tilde{y}'(s) \phi(\tilde{t}(s)) ds = \int_0^T \phi d \dot{y}
 \end{align*}
 so $\dot{y}_k \rightharpoonup^* \dot{y}$. Since $y_k(0) = y(0) = y_0$ this also implies $y_k \to y$ in $L^1([0,T];\R^n)$ and so $(y,\dot{y}) \in W^1_\mu([0;T];\R^n)$.

 For the converse take the sequence $y_k \to y$ from the definition of $W^1_\mu([0,T];\R^n)$. Let $\tilde{t}_k,\tilde{y}_k$ be associated relaxed curves defined on intervals $[0,S_k]$. We can assume those to be normalized in the sense that $\max(\tilde{t}_k',\abs{\tilde{y}_k'}) = 1$ almost everywhere, by taking $A(t,y,u) = u$ (and thus $\tilde{A}(\tilde{t},\tilde{y},\tilde{v},\tilde{u}) = \tilde{u}$) and applying \autoref{lem:uniqNormalization}. 
 
 Similarly, we temporarily choose the energy $\mathcal{F}(y,u) := \int_0^T \abs{u} dt = \int_0^T \abs{\dot{y}} dt$. Then as the $\dot{y}_k$ converge in the sense of measures and the energy corresponds to their total variation, it is in particular bounded along this sequence. But then the same holds true for the relaxed energy for the relaxed curves and we can apply \autoref{lem:uniformBounds} to get uniform bounds.\footnote{All of this can of course be done directly, without defining $A$ and $\mathcal{F}$, by repeating the arguments used before. However there is no new insight in this.} We then extend the curves to a common interval $[0,S]$ and extract limits and a converging subsequence, as in the proof of \autoref{thm:existConv}.
 
 Now $\tilde{y}_k \to \tilde{y}$ and $\tilde{t}_k \to \tilde{t}$ uniformly. But then for all $s \in [0,S]$ for which $\tilde{t}$ is injective we have
 \[y(\tilde{t}(s)) \leftarrow y_k(\tilde{t}(s)) = \tilde{y}_k(s) \to \tilde{y}(s)\]
 which is the first line. Furthermore we have by change of variables
 \[\int_0^{S_k} \tilde{y}_k'(s) \phi(\tilde{t}_k(s)) ds = \int_0^{S_k} \frac{d y_k}{dt}(\tilde{t}_k(s)) \frac{d\tilde{t}_k}{ds}(s) \phi(\tilde{t}_k(s)) ds = \int_0^T \dot{y}_k(t) \phi(t) dt.\]
 Here the right hand side converges to $\int_0^T \phi d\dot{y}$ by the weak* convergence of $y_k$ and the left hand side converges to $\int_0^S \tilde{y}'(s) \phi(\tilde{t}(s)) ds$ by the uniform convergence which proves the second line.
\end{proof}

\begin{remark}[On loss of information (part 1)]
 The non-uniqueness of the converse is related to the path that $y$ takes in the vertical parts. As the relaxed curve is just a curve in space-time, it contains the full information about this path in its trace. On the other hand, $(y,\dot{y}) \in W^1_\mu([0,T];\R^n)$ sees only the corresponding jump. It does not care, if $y$ moved along a straight line or took a longer detour. Even in an optimal control problem on $\R^n$, $y$ might not always take the straight line (see \autoref{ex:updown}). However, things get even more interesting if one thinks of problems where $y$ is restricted to a submanifold which is not simply connected such as $\S^1$. In this case the path taken also matters in a topological sense. We will discuss this in a bit more detail in \autoref{rem:lossOfInformation2}.
 
 Also note that this is the only source of non-uniqueness. If $\tilde{y}$ is assumed to have no vertical parts, or equivalently if $\tilde{t}$ is taken to be strictly monotonous, then the correspondence of $(y,\dot{y})$ and $(\tilde{t},\tilde{y})$ is one to one.
\end{remark}

We have thus seen how without oscillations the vertical relaxed curves $\tilde{y}$ can be embedded in $W^{1}_\mu$. As the derivative of $\tilde{y}$ is only used linearely, the change to generalized curves has no impact here and the result holds just the same. What is left is to discuss how the Young measure $\mu$ relates to the DiPerna-Majda measures used by Kružík and Roubíček. For this we recapitulate the definition.

Let $\gamma \R^k$ denote the compactification of $\R^k$ created by adding infinite directions.
A DiPerna-Majda measure on $[0,T]$ now consists of a positive Radon measure $\sigma \in M([0,T])$ and a   $\sigma$-measurable family of Radon measures $\mu_t \in M(\gamma \R^k)$. The general idea here is that $\mu_t$ behaves a bit like an ordinary Young measure, however allowing for concentrations using $\sigma$. Specifically a DiPerna-Majda measure $(\sigma,\mu_t)$ is generated by a bounded sequence $u_k \in L^p([0,T];\R^k)$ if for any $g\in C([0,T])$ and $w \in C(\gamma \R^k)$

\[\int_0^T g(t) w(u_k(t)) (1+\abs{u_k(t)}^p) dt \to \int_0^T g(t) \int_{\gamma \R^k} w(y) d\mu_t(y) d \sigma(t)\]
and any bounded sequence has a subsequence generating such a measure. It is related to our vertical relaxation by the following procedure.

\begin{proposition}
 Let $u_k \in L^1([0,T];\R^k)$ be a bounded sequence and $\tilde{v}_k,\tilde{u}_k \in L^\infty([0,S_k];\R^k)$ the corresponding normalized sequence of vertical relaxations. Let $\nu_s$ be the Young measure generated by $\tilde{v}_k,\tilde{u}_k$ and $\tilde{t}:[0,S]\to[0,T]$ the corresponding time. Then $u_k$ generates the DiPerna-Majda measure $(\sigma,\mu_t)$ where $\sigma := \tilde{t}_\star (\rho ds)$ is the pushforward of the measure $\rho(s) ds$ on $[0,S]$ and $\mu_t$ is the average of $\frac{1}{\rho}(1+\abs{.})\nu_s$ over all $s$ that map to the same time $t$. To be more precise, $\sigma$ and $\mu_t$ are defined such that for any $g \in C([0,T])$ and any $w \in C(\gamma\R^k)$ we have
 \begin{align*}
  \int_{[0,T]} g(t) d\sigma(t) &:= \int_{[0,S]} \rho(s) g(t(s)) ds &&\\
  \int_{\gamma \R^k} w(y)  d \mu_t(y) &:= \begin{cases}\int_{\R^k} \frac{1}{\rho(\tilde{t}^{-1}(t))} w(a/b)(b+\abs{a}) d\nu_{\tilde{t}^{-1}(t)}(a,b) &\text{ if $\tilde{t}^{-1}(t)$ is a point}\\
   \fint_{\tilde{t}^{-1}(t)} \frac{1}{\rho(s)} \int_{\R^k} w(a/b)(b+\abs{a}) d\nu_s(a,b) ds &\text{ if $\tilde{t}^{-1}(t)$ is an interval.}
  \end{cases}
 \end{align*} 
 where $\rho(s):=\int(\abs{a}+\abs{b})d\nu_s(a,b)$ is necessary to normalize $\mu_t$ to unit mass.
 Note that, as $\tilde{t}$ is monotone and surjective, the last two are the only two possibilities.
\end{proposition}
 
\begin{proof} 
 Since $\tilde{v}_k,\tilde{u}_k$ are normalized, $\supp \nu_s$ can only include points $(a,b)$ for which $1\leq \abs{a}+\abs{b} \leq 2$. Thus $1\leq \rho(s) \leq 2$ and as a consequence the measures $\sigma$ and $\mu_t$ are well defined measures by the above relations. Furthermore per definition $\mu_t(\gamma\R^k)=1$ for all $t\in[0,T]$.
 
 To see that they form a DiPerna-Majda measure, we only need to check if they are generated by the sequence $u_k$. For this let $g\in C([0,T])$ and $w \in C(\gamma \R^k)$ be fixed. Then by the definition of the relaxation
 \begin{align*}
  \phantom{{}={}}\int_0^T g(t) w(u_k(t)) (1+\abs{u_k(t)}) dt &= \int_0^{S_k} g(\tilde{t}(s)) w\left(\frac{\tilde{u}_k(s)}{\tilde{v}_k(s)}\right) \left(1+\abs{\frac{\tilde{u}_{k}(s)}{\tilde{v}_k(s)}}\right) \tilde{v}_k(s) ds\\
  = \int_0^{S_k} g(\tilde{t}(s)) w\left(\frac{\tilde{u}_k(s)}{\tilde{v}_k(s)}\right) \left(\tilde{v}_k(s)+\abs{\tilde{u}_{k}(s)}\right) ds 
  &\to \int_0^S g(t(s)) \int_{\R^k} w\left(\frac{a}{b}\right)(b+\abs{a}) \,d \nu_s(a,b) ds
  \end{align*}
  where $w(a/0) = \lim_{b\searrow 0} w(a/b)$ is a well defined abuse of notation. Continuing we get
  \[=\int_0^S g(t(s)) \rho(s) \int_{\R^k} \frac{1}{\rho(s)} w\left(\frac{a}{b}\right)(b+\abs{a}) \,d \nu_s(a,b) ds. \]
  Now we push everything forward by the function $\tilde{t}$ to get, as per above definitions
  \[=\int_0^T g(t) \int_{\gamma\R^k} w(y) \, d\nu_t(y) \sigma(t) \qedhere\]
\end{proof}

\begin{remark}[On the loss of information (part 2)] \label{rem:lossOfInformation2}
 The above process is surjective. Any DiPerna-Majda measure is generated by a sequence, which in turn has at least a subsequence generating a vertically relaxed Young measure for which the above process will produce the initial DiPerna-Majda measure. It is however not injective. While the produced DiPerna-Majda measure is unique, different Young measures may generate the same DiPerna-Majda measure.
 
 The specific problem is that the information about the order in which values are taken in the vertical part is lost. Take a look back at example \autoref{ex:updown}.
 In the concentration around $t=1/2$, the value of $u$ first went towards $+\infty$ and then towards $-\infty$. We could have constructed a similar sequence with the opposite order. Both will converge to the same DiPerna-Majda measure but not the same space-time relaxation. In other words, the space-time relaxation is finer than the one via DiPerna measures, in the sense of convex compactifications. (See \cite{roubicekBook} for a general discussion)
 
 In this case both ways would have had the same result in the limit, i.e. the same energy and the same curve $y(t)$ for almost all $t$. However it is not hard to think of an example where the curve would end up in different final positions depending on the order in which $u$ takes its values.
\end{remark}

\section{Banach-space valued problems} \label{sec:Banach}

It is not fundamentally difficult to extend the preceding discussions to functions that take their values in general Banach spaces, i.e.\@ we can consider problems of the kind

\begin{align*}
 \text{Minimize } & \mathcal{F}(y,u) := \int_0^T f(t,y(t),u(t)) dt + g(y(T))\\
 \text{where } & \dot{y}(t) = A(t,y(t),u(t)) \forall t\in [0,T]\\
 &y(0) = y_0
\end{align*}

where $f: [0,T] \times X_1 \times X_2 \to \R$, $g:X_1 \to \R$ and $A: [0,T] \times X_1 \times X_2 \to X_1$ and $X_1$ and $X_2$ are Banach-spaces.

The key observations here are that the vertically relaxed problem is constructed mostly using simple algebraic transformations and that in the relaxed problem we only need to consider bounded functions $(\tilde{v},\tilde{u}):[0,S] \to \R^+ \times X_2$ and Lipschitz-curves $(\tilde{t},\tilde{y}): [0,S] \to [0,T] \times X_1$, both of which apart from the usual problems regarding compactness do not behave fundamentally different to their finite-dimensional counterparts\footnote{For a way to define Young measures on Banach-spaces, see for example the works of Ahmed \cite{ahmedPropertiesRelaxedTrajectories1983} and Fattorini \cite{fattoriniRelaxedControlsInfinite1991,fattoriniInfiniteDimensionalOptimization1999}, to which we will refer for details.}. As such a general theory will not offer much more insight into the problem and only overcomplicate things with notation. Instead, we will illustrate the usefulness of the approach in an example.

\begin{example}
Let us discuss as an example an adaptation of \autoref{ex:updown} into a PDE-problem. Specifically, we want to look at the classic heat equation on a circle, where we control the right hand side. So let $\T = \R/\Z$ and define $X_1 :=  \dot{W}^{2,2}(\T) = \{y\in W^{2,2}(\T): \int_\T y dx = 0\}$ and $X_2 := \dot{L}^2(\T):= \{u\in L^2(\T):\int_\T u dx= 0\}$ as spaces with zero mean. Now consider the differential equation
\[\begin{cases}\partial_t y = \Delta y + u & \text{ on } [0,T] \times \T\\ y(0,.) = y_0 = 0 &\text{ on } \T.\end{cases}\]
where $t\mapsto u(t,.), [0,T] \to \dot{L}^2(\T)$ is our control parameter.
If we take the energy to be
\[\mathcal{F}(y,u) := \int_0^T \norm[L^2]{u(t,.)}dt + \norm[L^2]{y(T,.)-y_f}^2\]
where $y_f \in \dot{W}^{2,2}(\T)$ is some fixed non-zero function, then this defines an optimal control problem. 

Intuitively the problem behaves as follows: We need to use $u$ to force $y$ to move towards $y_f$. However since the heat equation tends to undo our progress by moving back towards its average and the control parameter is rate independent, it is advantageous to do so at the very last moment.

Let us start with a lower estimate. Let $S$ be the solution-operator of the corresponding Cauchy problem, i.e. the linear operator such that $\partial_t (S(t)y_0) = \Delta (S(t)y_0)$ and $S(0)y_0 = y_0$ for all $y_0 \in L^2(\T)$. Then by Duhamel's principle, for any given $u:[0,T] \to \dot{L}^2(\T)$, the differential equation is solved by
\[y(t) := \int_0^t S(t-s) u(s) ds.\]
Now it is well known that the semigroup associated with the heat kernel is contracting, specifically in this case $\norm[L^2]{S(t)u_0} \leq \norm{u_0}$ with equality only if $t=0$ or $u_0$ is constant. But then
\begin{align} \label{eq:exBanach:lower}
\norm[L^2]{y(T)} \leq \int_0^T \norm[L^2]{S(T-s)u(s)} ds \leq \int_0^T \norm[L^2]{u(s)} ds
\end{align}
which gives us the lower bound
\[\mathcal{F}(y,u) \geq \norm[L^2]{y(T)} + \norm[L^2]{y(T)-y_f}^2.\]

We can also show that this bound is optimal. Fix $y_T\in \dot{W}^{2,2}(\T)$. We define the following sequence:

\[u_k(t,.) := \dot{\varphi}_k(t) y_T - \varphi_k(t) \Delta y_T \]
where $\varphi_k:[0,T] \to [0,1]$ monotone with $\varphi_k(0) = 0$ and $\varphi_k(T) = 1$.
Then a short calculation reveals that
\[y_k(t,.) := \varphi_k(t) y_T\]
solves the corresponding differential equation with $y_k(T,.) = y_T$. Now we have
\[\int_0^T \norm[L^2]{u_k(t)} dt \leq \int_0^T |\dot{\varphi_k}| \norm[L^2]{y_T} + \abs{\varphi_k} \norm[L^2]{\Delta y_T} dt = \norm[L^2]{y_T}+\norm[L^1]{\varphi_k} \norm[L^2]{\Delta y_k}\]
so if we pick $\varphi_k$ such that $\supp \varphi_k \subset [T-1/k,T]$, then
\[\mathcal{F}(y_k,u_k) \to \norm[L^2]{y_T} + \norm[L^2]{y_T-y_f}^2\]
which, as before, is in turn optimized by choosing $y_T = \frac{1}{2}y_f$.

A careful look at the arguments shows that the optimum can never be obtained. As the approximating sequence suggests, $u$ needs to concentrate at $t=T$. If $u$ is not equal to $0$ almost everywhere in $[0,T)$, then \eqref{eq:exBanach:lower} will be a strict inequality. But as a minimizer needs to satisfy $y(T)=\frac{1}{2}y_f\neq 0$, this is impossible with a classical u.

So now let us turn to the relaxation. The procedure is the same as before. We replace $u$ with $\frac{\tilde{u}}{\tilde{v}}$, rewrite everything in terms of $s$ and add the additional equation $\tilde{t}' = \tilde{v}$. This results in the equations
\begin{align*}
 \partial_s \tilde{y} = - \tilde{v} \Delta_x \tilde{y} + \tilde{u} &\text{ on } [0,S] \times \T \\
 \tilde{y}(0,.) = 0 &\text{ on } \T\\
 \tilde{t}' = \tilde{v} &\text{ on } [0,S], \tilde{t}(0) = 0, \tilde{t}(S) = T
\end{align*}

and the energy
\begin{align*}
 \tilde{\mathcal{F}}(\tilde{y},\tilde{u}) & := \int_0^S \norm[L^2]{\tilde{u}} ds + \norm[L^2]{\tilde{y}(S)-y_f}^2.
\end{align*}

In this relaxation, the minimizer concentrating at $t=T$ exists and is easy to construct. Let $y_T := \frac{y_f}{2}$ as before and define $L:= \norm[L^2]{y_T}$, $S:= T+L$. Now we pick
\begin{align*}
 (\tilde{v},\tilde{u})(s) := \begin{cases} (1,0) &\text{ for } 0 \leq s \leq T \\ \left(0,\frac{y_T}{L}\right) &\text{ for } T < s \leq T+L \end{cases}
\end{align*}

Then simple integration shows that our solution will look like
\begin{align*}
 (\tilde{t},\tilde{y})(s) = \begin{cases}
                             (s,0) &\text{ for } 0 \leq s \leq T \\
                             \left(T,(s-T)\frac{y_T}{L}\right) &\text{ for } T < s \leq T+L 
                            \end{cases}
\end{align*}
which indeed satisfies $y(T) = y_T$ and $\int_0^S \norm[L^2]{\tilde{u}} ds = \norm{y_T}$. It is also possible to see that this is indeed a minimizer by deriving a modified Duhamel-formula
\begin{align*}
 \tilde{y}(s) := \int_0^s S(\tilde{t}(s)-\tilde{t}(r)) \tilde{u}(r) dr
\end{align*}
and adapting the preceeding arguments.
\end{example}

\bibliographystyle{alpha}
\bibliography{bibliography} 

\end{document}